\documentclass[11pt,leqno]{amsart}
\usepackage{amsmath,amsthm}
\usepackage{amsfonts,amssymb}
\newtheorem{theorem}{Theorem}

\newtheorem{proposition}[theorem]{Proposition}
\newtheorem{corollary}[theorem]{Corollary}

\def\P{{\mathcal P}}
\def\Q{{\mathcal Q}}
\def \sh{{\rm sh \:}}
\def \sgn{{\rm sgn }}

\title[Determinantal transition kernels]{Determinantal transition kernels for some interacting particles on the line}
\author{A. B. Dieker}
\address{University College Cork, Probability Group, 17 South Bank, Crosses Green, Cork, Ireland}
\email{t.dieker@proba.ucc.ie}
\author{J. Warren}
\address{University of Warwick, Department of Statistics, Coventry, CV4 7AL, United Kingdom}
\email{j.warren@warwick.ac.uk}

\begin{document}
\begin{abstract}
We find the transition kernels for four Markovian interacting particle systems on the line,
by proving that each of these kernels is intertwined with a Karlin-McGregor type kernel.
The resulting kernels all inherit the determinantal structure from the
Karlin-McGregor formula, and have a similar form to Sch\"utz's kernel for
the totally asymmetric simple exclusion process.

\vspace{5mm}

\noindent {\sc Resum\'e.} Nous trouvons les noyaux de transition de quatre syst\`emes markoviens de particules en interaction sur une ligne, en prouvant que chacun de ces noyaux s'entrelace avec un noyau du type de Karlin-McGregor.  Tous les noyaux r\'esultants  h\'eritent de la structure de d\'eterminant de la formule de  Karlin-McGregor et ont une forme  similaire \`a celle du noyau  de Sch\"utz pour le processus d'exclusion simple totalement asym\'etrique.
\end{abstract}

\keywords{Interacting particle system, intertwining, Karlin-McGregor theorem, Markov transition kernel, Robinson-Schensted-Knuth correspondence, Sch\"utz theorem, stochastic recursion, symmetric functions.}
\subjclass[2000]{%
Primary:
60J05, 
60K35, 
05E10; 
Secondary:
05E05, 
15A52. 
}

\maketitle
\section{Introduction}
Non-colliding Markov processes are canonical examples of stochastic processes with a determinantal transition kernel,
given by the Karlin-McGregor formula.
A determinantal transition kernel of a different form, yet similar to the Karlin-McGregor
kernel, was encountered by Sch\"utz~\cite{schuetz:exclusion1997}
in his study of the totally asymmetric simple exclusion process.
This work has stimulated much recent research, e.g.,
\cite{borodin:fluctuation2007,johansson:meixner2007,povolotskypriezzhev:parallel2006,rakosschutz:bethe2006,tracywidom:fluctuation2007,warren:dyson2007}.

In this note we explicitly connect Sch\"utz type formulae for
particle systems and Karlin-McGregor type formulae for non-colliding processes.
Our approach builds upon deep connections between particle processes and non-colliding processes (or random-matrix theory),
which have been recently discovered \cite{baryshnikov:guequeues2001,draief:queuesstores2005,johansson:shapefluc2000,oconnell:pathRS2003}.
A combinatorial correspondence known as the Robinson-Schensted-Knuth (RSK) correspondence links these processes.
This RSK correspondence gives a {\em coupling} of the non-colliding
process and the particle process.
Our key contributions are that this coupling implies an {\em intertwining} of their transition semigroups,
and that this intertwining is enough to find the transition kernel of
the particle process.

We give a number of new formulae by applying this method to four variants of the RSK correspondence.
Augmented with systems arising from suitable limiting procedures, the particle systems we treat in this way
are known to play pivotal roles in a wide range of interesting applied problems.
For instance, they appear in the context of queues in series, last-passage percolation,
growth models, and fragmentation models.

This note is organized as follows. In Section~\ref{sec:interacting}, we introduce four
interacting particle systems and we present the associated transition kernels.
Section~\ref{sec:RSK} describes the four variants of the RSK correspondence we use
in our analysis, and derives the aforementioned intertwining of the semigroups.
Finally, it is the topic Section~\ref{sec:determinant} to use this intertwining
for finding the transition mechanism of the interacting particles.

\section{Interacting particles on the line; main results}
\label{sec:interacting}
We are concerned with  a system of $N$ particles, each with a
position in the integer lattice ${\mathbf Z}$, evolving in discrete
time. We will consider four possible  cases. Particles will move
from the left to the right  making either Bernoulli or geometrically
sized jumps, with one of two possible interactions that maintains their relative orderings (blocking
or pushing). We begin by
describing these four processes more precisely. In each case
$Y_i(n)$ denotes the position of particle number $i$ at time $n$.
We order the particles, so that $Y$ takes values in either $W^N$
or $\hat W^N$, where
\begin{eqnarray*}
W^N&=&\bigl( z \in {\mathbf Z}^N; z_N\le z_{N-1}\le \ldots\le z_1\bigr)\\
\hat W^N&=&\bigl( z \in {\mathbf Z}^N; z_1\le z_2\le \ldots\le z_N\bigr).
\end{eqnarray*}
Throughout, we let $p=(p_1,p_2,\ldots ,p_N)$ be a vector with each $p_k\in (0,1)$.

\medskip
{\bf CASE A: Geometric jumps with pushing.} Particles are labelled
from left to right, so $Y_1(n) \leq Y_2(n) \leq \ldots \leq Y_N(n)$.
Between time $n-1$ and $n$, each of the particles moves to the right
according to some geometrically distributed jump, having parameter $p_i$ for particle $i$. The order in which
the particles jump is given by their labels, so the leftmost
particle jumps first. Overtaken particles (if any) are moved to the same
position as the jumping particle, a position from which the next particle subsequently makes its own jump.
One can thus think of particles `pushing' other particles to maintain their relative orderings.

This leads to the following stochastic recursion.
The evolution is generated from a family $\bigl( \xi(k,n);
k\in \{1,2,\ldots N\}, n \in {\mathbf N} \bigr)$ of independent
geometric random variables satisfying ${\mathbf
P}(\xi(k,n)=r)=(1-p_k)p_k^r$ for $r=0,1,2, \ldots$, via the
recursions $Y_1(n)=Y_1(n-1)+\xi(1,n)$, and for $k=2,3,\ldots N$,
\[
Y_k(n)=\max\bigl(Y_k(n-1), Y_{k-1}(n)\bigr)+\xi(k,n).
\]
Note that $Y=\bigl(Y(n);n\ge 0\bigr)$ is a Markov chain on $\hat W^N$.

One application area where this recursion arises is the theory of queueing networks.
Indeed, the particle system with exponentially distributed jumps, which is obtained after a suitable
limiting procedure, relates to a series Jackson network.
Here $Y(n)$ corresponds to the departure instants of the $n$th customer from each of $N$ queues in series.
These networks are investigated further in our companion paper~\cite{diekerwarren:jackson2007}.

The vector $Y$ also plays an important role in the context of directed last-passage percolation
with geometrically distributed travel times and `origin' $(1,1)$, where
$Y(n)$ can be interpreted as the vector of maximal travel times
to the sites $(n+1,1),\ldots,(n+1,N)$.
Very recently, Johansson~\cite{johansson:meixner2007} has derived the transition kernel
of $Y$ in the case of equal rates $p_1=\ldots=p_N$, with different methods than presented here.

\medskip
{\bf CASE B: Bernoulli jumps with blocking.} Particles are labelled
from right to left, so $Y_N(n) \leq Y_{N-1}(n) \leq \ldots \leq
Y_1(n)$. Between time $n-1$ and $n$ each particle attempts to move
one step to the right, but it is constrained
not to overtake the particle to its right. Particle $i$ moves with probability $p_i$. 
The particles are now updated from right to left, so it is the updated
position of the particle to the right that acts as a block.

The evolution is generated from a family $\bigl( \xi(k,n); k\in \{1,2,\ldots N\}, n
\in {\mathbf N} \bigr)$  of independent Bernoulli random variables
satisfying ${\mathbf P}(\xi(k,n)=+1)=1-{\mathbf P}(\xi(k,n)=0)=p_k$,
via the recursions $Y_1(n)=Y_1(n-1)+\xi(1,n)$, and for $k=2,3,\ldots
N$,
\[
Y_k(n)=\min\bigl(Y_k(n-1)+\xi(k,n), Y_{k-1}(n)\bigr).
\]
In particular, $Y$ is a Markov chain on $W^N$.

The process $Y$ has been investigated by R{\'a}kos and Sch{\"u}tz~\cite{rakosschutz:fragmentation2005}
in the context of a fragmentation model.
On shifting the $i$-th particle $i$ positions to the left, $Y$ corresponds to the discrete-time 
totally asymmetric simple exclusion process (TASEP) with sequential updating.
Moreover, the process arises in a directed first-passage percolation model known as the Sepp\"al\"ainen model~\cite{seppalainen:shape1998} with `origin' $(0,0)$,
where $Y(n)$ corresponds to the vector of instants at which the sites $(n,0),\ldots,(n,N-1)$ become wet.

An important process arises if we scale $Y$ as in the law of small numbers, i.e.,
by setting $p_i=\alpha_i/M$ and considering $Y(\lfloor Mt\rfloor)$ for $t\in\mathbf R_+$ as $M\to\infty$.
In the case of equal rates, the resulting continuous-time Markov process describes, after a deterministic shift,
the positions of $N$ particles in the (continuous-time) TASEP. 
This is the framework originally studied by Sch\"utz~\cite{schuetz:exclusion1997}, who derives
the transition kernel of this process. It has recently been extended to particles hopping at different
rates by R{\'a}kos and Sch{\"u}tz~\cite{rakosschutz:bethe2006}.

The same continuous-time process is also of significant importance for series Jackson queueing networks
as well as for a corner-growth model.
In the queueing context, it represents the cumulative number of departures from each of the queues;
see \cite{diekerwarren:jackson2007}.
In the corner-growth model, it represents the height of the first $N$ columns.
We refer to K\"onig's survey paper~\cite{konig:orthogonal2005} for these and further connections,
such as the relation between the $N$th component of this process
and directed last-passage percolation with exponentially distributed travel times.

\medskip
{\bf CASE C: Geometric jumps with blocking.} Once again particles
are labelled from right to left,  so $Y_N(n) \leq Y_{N-1}(n) \leq
\ldots \leq Y_1(n)$. Between time $n-1$ and $n$ each particle
attempts to move a geometrically distributed number of steps to the
right, starting with the leftmost particle. As in case B, a particle
is blocked by the closest particle
to its right if it tries to overtake another particle,
but this time it is the old position of that particle that acts as a block.

The evolution is generated from a family $\bigl( \xi(k,n); k\in
\{1,2,\ldots N\}, n \in {\mathbf N} \bigr)$  of independent
geometric random variables satisfying ${\mathbf
P}(\xi(k,n)=r)=(1-p_k)p_k^r$ for $r=0,1,2, \ldots$, via the
recursions $Y_1(n)=Y_1(n-1)+\xi(1,n)$, and  for $k=2,3,\ldots N$,
\[
Y_k(n)=\min\bigl(Y_k(n-1)+\xi(k,n), Y_{k-1}(n-1)\bigr).
\]

As laid out by Draief {\em et al.}~\cite{draief:queuesstores2005},
this recursion arises in the study of so-called tandem stores in series.
The random vector $Y$ represents the
cumulative demand met at each of the stores when the first store is saturated.
Alternatively, $Y$ can be interpreted as the minimum-weight vector of
certain weighted lattice paths.

\medskip
{\bf CASE D: Bernoulli jumps with pushing.} Now it is natural to
label  particles left to right, so $Y_1(n) \leq Y_2(n) \leq \ldots
\leq Y_N(n)$. We update from right to left, and preserve the
ordering by pushing particles to the right.
The evolution is
generated from a family $\bigl( \xi(k,n); k\in \{1,2,\ldots N\}, n
\in {\mathbf N} \bigr)$ of independent Bernoulli random variables
satisfying ${\mathbf P}(\xi(k,n)=+1)=1-{\mathbf P}(\xi(k,n)=0)=p_k$,
via the recursions $Y_1(n)=Y_1(n-1)+\xi(1,n)$, and for $k=2,3,\ldots
N$,
\[
Y_k(n)=\max\bigl(Y_k(n-1)+\xi(k,n), Y_{k-1}(n)\bigr).
\]
The process $Y$ is the discrete-time analogue of a particle system studied by Alimohammadi {\em et al.}~\cite{alimohammadi:exact1998}, which has been studied further by Borodin and Ferrari~\cite{borodinferrari:large2007}.
It also plays an important role in the directed {\em last}-passage analogue of the Sepp\"al\"ainen model.

\medskip
We summarize the description of the four cases in the following table.

\begin{table}[h]
\begin{tabular}{|l|c|c|c|c|}
\hline
\textbf{Case} & A & B& C& D\\
\hline
\hline
\textbf{Jump distribution} & geometric & Bernoulli & geometric & Bernoulli \\
\hline
\textbf{Interaction} & pushing & blocking & blocking & pushing \\
\hline
\textbf{Updating} & from left & from right & from left & from right\\
\hline
\end{tabular}
\end{table}

We need some well-known symmetric functions in order to present
the Markov transition kernel of $Y$ in each of the four cases.
The $r$th complete homogeneous
symmetric polynomials in the indeterminates $\alpha_1,\ldots
\alpha_N$ is given by
\[
h_r(\alpha)=\sum_{k_1\geq 0, \ldots, k_N \geq 0; k_1+k_2+\cdots +k_N=r}
\alpha_1^{k_1}\alpha_2^{k_2} \cdots \alpha_N^{k_N}.
\]
By convention $h_0=1$ and $h_r=0$ for $r<0$.  Now for $1\leq i <
j \leq N$, let $h^{(ij)}_r(\alpha) = h_r(\alpha^{(ij)})$ where
$\alpha^{(ij)}$ is the $N$-vector
$(0,\ldots,0,\alpha_{i+1},\alpha_{i+2}, \ldots ,\alpha_j,0, \ldots
0)$ obtained from $\alpha$ by setting the first $i$ weights, and the
last $N-j$  weights  equal to $0$. Equivalently it is the $r$th
complete homogeneous symmetric polynomial in the indeterminates
$\alpha_{i+1}, \ldots, \alpha_j$.
We set $h^{(jj)}_r(\alpha)= {\mathbf 1}(r=0)$.

We also need $e_r$, the $r$th elementary symmetric function defined as
\[
e_r(\alpha)= \sum _{k_1<k_2< \cdots <k_r} \alpha_{k_1} \cdots \alpha_{k_r}.
\]
In analogy with the complete homogeneous symmetric functions, we
use the conventions $e^{(jj)}_r(\alpha)={\mathbf 1}(r=0)$ and $e_0=1$.
We also set $e^{(ij)}_r(\alpha)= e_r(\alpha^{(ij)})$ and $e_r=0$ for $r<0$.

Given a function $f$ on $\mathbf Z$ and a vector $\alpha\in \mathbf R^N_+$, we write
\begin{equation}
\label{eq:defwij}
f^{(ij)}_\alpha(k)= \begin{cases}
 \sum_{\ell=0}^{i-j} (-1)^\ell e^{(ji)}_\ell(\alpha) f(k+\ell) & \text{ if  }j \leq i, \\
 \sum_{\ell=0}^\infty h^{(ij)}_{\ell} (\alpha) f(k+\ell)& \text{ if }  i\leq j,
\end{cases}
\end{equation}
and
\[
\hat f^{(ij)}_{\alpha}(k)= \begin{cases}
 \sum_{\ell=0}^{i-j} (-1)^\ell e^{(ji)}_\ell(\alpha) f(k-\ell) & \text{ if  }j \leq i, \\
 \sum_{\ell=0}^\infty h^{(ij)}_{\ell} (\alpha) f(k-\ell)& \text{ if }  i\leq j,
\end{cases}
\]
provided all series converge absolutely.
Our main theorem uses this notation for functions $f$ belonging to two different
families, $\bigl(w_n; n\in \mathbf Z\bigr)$ and $\bigl(v_n; n\in \mathbf Z\bigr)$, for which the desired
convergence holds. 
These families are defined through
\[
w_n(k)=\binom {n-1+k}{k}{\mathbf 1}(k\ge 0; n\ge 0) \quad\text{and}\quad
v_n(k)=\binom {n}{k}{\mathbf 1}(0\le k\le n; n\ge 0).
\]
We write $w^{(ij)}_{n,\alpha} $ for $f^{(ij)}_\alpha$ with $f=w_n$,
and define $\hat w^{(ij)}_{n,\alpha}$, $v^{(ij)}_{n,\alpha} $, and $\hat v^{(ij)}_{n,\alpha}$ similarly.
Moreover, we abbreviate the vector $(p^{-1}_1,\ldots,p^{-1}_N)$ by $p^{-1}$,
and define the vector $\pi$ through $\pi_i=p_i/(1-p_i)$.

\begin{theorem}
The transition kernel $Q_n$ of the process $Y$ is given by the following expressions.

\medskip
{\bf CASE A: Geometric jumps with pushing.}
We have for $y,y'\in \hat W^N$,
\[
Q_n(y,y')=\prod_{k=1}^N\left[(1-p_k)^n p_k^{y_k'-y_k}\right]
\det\left\{\hat w^{(ij)}_{n,p^{-1}}(y'_i-y_j+i-j)\right\}.
\]

\medskip
{\bf CASE B: Bernoulli jumps with blocking.}
We have for $y,y'\in W^N$,
\[
Q_n(y,y')=\prod_{k=1}^N\left[(1-p_k)^n \pi_k^{y_k'-y_k}\right]
\det\left\{v^{(ij)}_{n,\pi} (y'_i-y_j-i+j)\right\}.
\]

\medskip
{\bf CASE C: Geometric jumps with blocking.}
We have for $y,y'\in W^N$,
\[
Q_n(y,y')=\prod_{k=1}^N\left[(1-p_k)^n p_k^{y_k'-y_k}\right]
\det\left\{w^{(ij)}_{n,p}(y'_i-y_j-i+j)\right\}.
\]

\medskip
{\bf CASE D: Bernoulli jumps with pushing.}
We have for $y,y'\in \hat W^N$,
\[
Q_n(y,y')=\prod_{k=1}^N\left[(1-p_k)^n \pi_k^{y_k'-y_k}\right]
\det\left\{\hat v^{(ij)}_{n,\pi^{-1}}(y'_i-y_j+i-j)\right\}.
\]
\end{theorem}

The remainder of this paper is devoted to a proof of this theorem.

\section{The RSK correspondence and its variants}
\label{sec:RSK}
In each of the four cases considered in the previous section, the
Markov process of interest $\bigl(Y(n); n \in{\mathbf N} \bigr)$ is constructed
from a family $\bigl( \xi(k,n); k\in \{1,2,\ldots, N\}, n \in
{\mathbf N} \bigr)$ of random innovations.
In this section we will construct a second Markov process $Z$ from
the same innovations data, using the RSK algorithm or one of its variants.
In each case, we will be able to find the (Karlin-McGregor type) transition semigroup of $Z$ and
show that it is intertwined with the transition semigroup of $Y$.

We give some definitions in order to describe the RSK type algorithms in the form we need.
A partition $\lambda$ with $k$ parts is an integer vector $\lambda_1,\ldots,\lambda_k$
satisfying $\lambda_1\ge \lambda_2\ge \ldots\ge \lambda_k$.
Consider an array of strictly positive integers $T=\bigl( T_{ij}; 1\le i\le k, 1\le j \le \lambda_i \bigr)$ 
of shape $\lambda$ satisfying $T_{ij}\le T_{i+1,j}$ and $T_{ij}\le T_{i,j+1}$ (interpret $T_{ij}$ as infinity if it is undefined). 
We write $T\in {\mathbf T}_N^{n,\wedge}$ if the integers in $T$ do not exceed $n$ and increase {\em strictly} down the columns, while $\lambda$ consists of at most $N$ parts. In the terminology of enumerative combinatorics, ${\mathbf T}_N^{n,\wedge}$ consists of semi-standard Young tableaux (SSYT) with at most $N$ rows and content $\{1,\ldots,n\}$.
Similarly, we write $T\in {\mathbf T}_N^{n,<}$ if the integers in $T$ do not exceed $n$ and increase strictly along the rows, while $\lambda$ consists of at most $N$ parts. We write $\sh(T)$ for the shape of $T$, which we consider to be an element of $W^N$ by padding the vector with zeros if necessary ($N$ is fixed throughout).

We study four different ways to associate a ${\mathbf T}_N^{N,\wedge}$-valued 
process $\bigl( \P(n); n\ge 0\bigr)$ to the data $\bigl( \xi(k,n); k\in \{1,\ldots,N\}, n\in {\mathbf N} \bigr)$, each corresponding
to a different variant of RSK.
More details on the different variants can for instance be found in \cite{fulton:youngtableaux1997,gansner:correspondence1981}.

We begin by noting two methods for constructing a two-line array of the form
\[
\left(\begin{array}{cccc}
a_1 & a_2 & a_3 & \cdots \\
b_1 & b_2 & b_3 & \cdots
\end{array}\right)
\]
from the innovation data, where the $a_i$ are nondecreasing. 
Both methods have the property that the column $\binom a b$ appears $\xi(b, a)$ times in the array.
The first method, {\em lexicographic array-construction}, requires that $b_i\le b_{i+1}$ if $a_i=a_{i+1}$.
The second method requires that $b_i\ge b_{i+1}$ if $a_i=a_{i+1}$, and we therefore call it the {\em anti-lexicographic array-construction}.

Next we describe two methods for constructing a sequence $\bigl(\P(n); n\ge 0\bigr)$ of SSYT in ${\mathbf T}_N^{N,\wedge}$
from the given two-line array. Both constructions are inductive, and start with an empty SSYT $\P(0)$.
Given $\P(n)$, the SSYT $\P(n+1)$ is found by inserting the elements $b_i$ for which $a_i=n+1$.
If there are $M$ such elements, the methods construct a sequence $\P^1(n),\ldots,\P^M(n)$ such that $\P^1(n)=\P(n)$ and
$\P(n+1)=\P^M(n)$.
The SSYT $\P^{i+1}(n)$ is constructed from $\P^{i}(n)$ by
inserting the next unused $b$-element of the two-line array.
The first method, {\em row insertion}, inserts an element $b$ into $\P^{i+1}(n)$ using the following rules:
\begin{itemize}
\item If every entry in the first row of $\P^i(n)$ is {\em smaller than or equal} to $b$, then $b$ is appended to the
end of the row.
\item Otherwise, $b$ is used to replace the leftmost entry in the row
which is {\em strictly larger} than $b$.
\end{itemize}
The entry replaced is inserted in the same manner into the second row,
and this process continues until an entry is either placed at the end of a row or it is placed in the first position of an empty row.
The second method, {\em column insertion}, is a modification of the above procedure. Instead of inserting entries along the rows,
the entries are now inserted down the columns.
Now the following rules are followed:
\begin{itemize}
\item If every entry in the column is {\em strictly smaller} than $b$, then $b$ is appended to the end of the column.
\item Otherwise, $b$ is used to replace the uppermost entry in the column which is {\em greater than or equal} to $b$.
\end{itemize}

Four combinations of $\xi$-values, array constructions, and insertion algorithms are of special interest,
as there is a combinatorial correspondence underlying the above construction of the process $\P$.
To explain this, let $\mathcal Q(n)$ be the unique array of integers
for which the entries $1,\ldots,m$ form an array with the same shape as $\mathcal P(m)$
for $m\le n$.
Depending on the RSK variant chosen, the pair $(\mathcal P(n),\mathcal Q(n))$ belongs either to
\begin{equation}
\label{eq:range1}
\{(S,T)\in {\mathbf T}_N^{N,\wedge}\times {\mathbf T}_N^{n,\wedge}: \sh(S)=\sh(T)\}
\end{equation}
or to 
\begin{equation}
\label{eq:range2}
\{(S,T)\in {\mathbf T}_N^{N,\wedge}\times {\mathbf T}_N^{n,<}: \sh(S)=\sh(T)\}.
\end{equation}
The four combinatorial correspondences map the sequence
$\bigl( \xi(k,m); k\in \{1,2,\ldots, N\}, m \in
\{1,2,\ldots,n\} \bigr)$ {\em bijectively} to $(\mathcal P(n),\mathcal Q(n))$ for any $n\ge 1$.
We present these four combinations in the next table, along with the name under which the resulting
bijection (correspondence) is known.

\begin{table}[h]
\begin{tabular}{|l|c|c|c|c|}
\hline
\textbf{Correspondence} & RSK & dual RSK & Burge & dual Burge \\
\hline
\textbf{Innovations data} & $0,1,2,\ldots$ &$0,1$ &$0,1,2,\ldots$ &$0,1$\\
\hline
\textbf{Range} \textbf{$(\P(n),\Q(n))$} & (\ref{eq:range1}) & (\ref{eq:range2}) & (\ref{eq:range1}) & (\ref{eq:range2})\\
\hline
\textbf{Insertion algorithm} & row & column & column & row \\
\hline
\textbf{Array construction} & lexicographic & lexicographic & anti-lexicographic & anti-lexicographic \\
\hline
\end{tabular}
\end{table}

It is our next aim to relate the four correspondences to the interacting-particle framework
of Section~\ref{sec:interacting}.
Further analysis is facilitated by
a second process $Z=\bigl(Z(n);n\ge 0\bigr)$ arising from the process $\P$.
We define this process by letting $Z(n)$ be the shape of $\P(n)$, so $Z$ takes values in $W^N$.
After giving the transition kernel of $Z$, we prove that the transition semigroup of $Y$
is intertwined with the transition semigroup of $Z$.

The semi-standard Young tableaux in ${\mathbf T}_N^{N,\wedge}$ play an important role in our analysis
since the process $\P$ is ${\mathbf T}_N^{N,\wedge}$-valued, and we need some further definitions for such tableaux.
For $T\in {\mathbf T}_N^{N,\wedge}$, we let
$\mathrm{ledge}(T)$ be the $N$-vector for which element $i$ is the number of $i$'s in row $i$ (the terminology `ledge' is motivated in the next section). This vector may contain zeros and we always have $\mathrm{ledge}(T)\in W^N$.
Similarly, we let $\mathrm{redge}(T)$ be the $N$-vector for which the $i$-th entry is the number of 
elements in the first row that do not exceed $i$; we always have $\mathrm{redge}(T)\in \hat W^N$.
For a vector $\alpha=(\alpha_1,\alpha_2, \ldots, \alpha_N)$ of weights,
we define the weight $\alpha^{T}$ of $T\in {\mathbf T}_N^{N,\wedge}$ by
\[
\alpha^{T}= \prod_{i=1}^N \alpha_i^{n_i(T)},
\]
where $n_i(T)$ is the number of $i$'s in $T$. This is the usual definition for the weight of a SSYT.

\medskip
{\bf CASE A: Geometric jumps with pushing.}
Consider the RSK algorithm, i.e., lexicographic array construction and row insertion.
The row-insertion algorithm shows that the vector-valued process $\bigl(\mathrm {redge} (\P(n)); n\ge 0\bigr)$
is {\em exactly the same} as the process $\bigl( Y(n);n\ge 0\bigr)$.
Note also that $Y$ is Markov relative to the filtration of the Markov process $\P$.

The dynamics of the RSK algorithm show that $\mathcal Q(n)$ is a SSYT.
Therefore, using the bijective property of RSK we obtain for $S\in {\mathbf T}_N^{N,\wedge}, 
T\in {\mathbf T}_N^{n,\wedge}$,
\begin{equation}
\label{eq:probbijection}
\mathbf P(\mathcal P(n)=S, \mathcal Q(n)=T)= \prod_{k=1}^N \left[1-p_k\right]^n p^S\,
\mathbf 1(\mathrm{sh} (S)=\mathrm{sh} (T)).
\end{equation}
Using the fact that $\mathcal Q(n)$ encodes the shapes $\bigl(Z(m); m\le n\bigr)$, 
as argued in \cite[Sec.~3.2]{oconnell:conditionedRSK2003} we can use (\ref{eq:probbijection}) to compute the law of
the shape process $Z$. 
Indeed, summing (\ref{eq:probbijection}) over appropriate $S\in {\mathbf T}_N^{N,\wedge}, 
T\in {\mathbf T}_N^{n,\wedge}$ we obtain for $m\le n$,
\[
\mathbf P(Z(n)=z(n),Z(m)=z(m),\ldots,Z(1)=z(1))= \prod_{k=1}^N \left[1-p_k\right]^n s_{z(n)}(p) g^{n-m}_{z(n)/z(m)},
\]
provided the left-hand side is nonzero.
Here $s_z(p)=\sum_{T\in {\mathbf T}_N^{N,\wedge};\, \mathrm{sh} (T)=z} p^T$ 
is a symmetric function (in $p$) known as a Schur polynomial. Also, $g^k_{\lambda/\mu}$ is the number of skew SSYT
with shape $\lambda/\mu$ and entries from $\{1,\ldots, k\}$ \cite[Sec.~7.10]{stanley:ec2}. 
The Jacobi-Trudi identity~\cite[Thm.~7.16.1]{stanley:ec2} implies
\[
g^k_{\lambda/\mu}=\det\left\{w_k(\lambda_i-\mu_j-i+j)\right\},
\]
and we thus find that $Z$ is a Markov chain on $W^N$ with transition kernel given by
\[
P_n(z,z')=
\prod_{k=1}^N \left[1-p_k\right]^n \frac{s_{z'}(p)}{s_z(p)} \det\left\{w_n(z'_i-z_j-i+j)\right\}.
\]
A similar reasoning, now summing (\ref{eq:probbijection}) over $S\in{\mathbf T}_N^{N,\wedge}$ with $\mathrm{redge}(S)=y\in\hat W^N$ and $\mathrm{sh}(S)=Z(n)$, shows that
\[
{\mathbf P}(Y(n)=y| Z(m),m\le n)= \hat K_p(Z(n), y),
\]
where
\[
\hat K_\alpha(z,y)= \frac 1{s_z(\alpha)}
\sum_{T\in\mathbf{T}_N^{N,\wedge}; \,\mathrm{sh}(T)=z, \,
\mathrm{redge}(T)=y} \alpha^{T}.
\]
Note that ${\mathbf P}(Y(n)=y| Z(m),m\le n)$ depends on $\bigl(Z(m);m\le n\bigr)$ only through $Z(n)$.
This yields for $y\in \hat W^N$ and $m\le n$,
\begin{eqnarray*}
{\mathbf P}(Y(n)=y| Z(k),k\le m) &=& {\mathbf P}\left[\left. {\mathbf P}(Y(n)=y| Z(k),k\le n) \right|Z(k),k\le m\right]\\
&=& {\mathbf P}\left[\left. \hat K_p(Z(n),y)  \right|Z(k),k\le m\right]\\
&=&  \sum_{z\in W^N} P_{n-m}(Z(m),z)\hat K_p(z,y).
\end{eqnarray*}
On the other hand, the left-hand side can be written as
\begin{eqnarray*}
{\mathbf P}(Y(n)=y| Z(k),k\le m) &=& {\mathbf P}\left[\left. {\mathbf P}(Y(n)=y| \P(k),k\le m) \right|Z(k),k\le m\right]\\
&=& {\mathbf P}\left[\left. Q_{n-m}(Y(m),y) \right|Z(k),k\le m\right]\\
&=& \sum_{y'\in \hat W^N} \hat K_p(Z(m),y')  Q_{n-m}(y',y).
\end{eqnarray*}
On combining these two displays, we deduce
the intertwining relationship $P_n \hat K_p=\hat K_p Q_n$,
where the product $P\hat K$ of the kernel $P$ with domain $W^N\times W^N$ and the kernel
$\hat K$ with domain $W^N\times \hat W^N$ is defined as
$P\hat K(z,y)=\sum_{z'\in W^N} P(z,z')\hat K(z',y)$.
The next section investigates the intertwining relationship in detail
to find the kernel $Q_n$.

\medskip
{\bf CASE B: Bernoulli jumps with blocking.}
Under lexicographic array construction and column insertion,
$\bigl(\mathrm{ledge} (\P(n)); n\ge 0\bigr)$
is exactly the same as the process $\bigl( Y(n);n\ge 0\bigr)$.
The bijective property of the dual RSK shows that for $S\in \mathbf T^{N,\wedge}_N,T\in\mathbf T^{n,<}_N$, 
\[
\mathbf P(\mathcal P(n)=S, \mathcal Q(n)=T)= \prod_{k=1}^N \left[1-p_k\right]^n \pi^{S}\,
\mathbf 1(\mathrm{sh} (S)=\mathrm{sh} (T)).
\]
As in \cite[Sec.~3.3]{oconnell:conditionedRSK2003}, on combining this
with the dual Jacobi-Trudi identity, we find that $Z$ is a Markov chain on $W^N$ with transition kernel
\[
P_n(z,z')=
\prod_{k=1}^N \left[1-p_k\right]^n \frac{s_{z'}(\pi)}{s_z(\pi)} \det\left\{v_n(z'_i-z_j-i+j)\right\}.
\]
After setting
\[
K_\alpha(z,y)= \frac 1{s_z(\alpha)}
\sum_{T\in\mathbf{T}^{N,\wedge}; \,\mathrm{sh}(T)=z, \, \mathrm{ledge}(T)=y} \alpha^{T},
\]
we may derive the intertwining $P_n K_\pi=K_\pi Q_n$ along the lines of case A.

\medskip
{\bf CASE C: Geometric jumps with blocking.}
The dynamics of the column insertion algorithm show that $\bigl(\mathrm {ledge} (\P(n)); n\ge 0\bigr)$
is exactly the same as the process $\bigl( Y(n);n\ge 0\bigr)$.
Moreover, since the law of the process $Z$ is invariant under the
choice of the insertion algorithm, $Z$ is Markovian with the same kernel as in case A.
The intertwining is $P_nK_p=K_p Q_n$.

\medskip
{\bf CASE D: Bernoulli jumps with pushing.}
Now $\bigl(\mathrm {redge} (\P(n)); n\ge 0\bigr)$
is exactly the same as the process $\bigl( Y(n);n\ge 0\bigr)$.
The process $Z$ is Markovian with the same kernel as in case B, and
we have the intertwining $P_n\hat K_\pi=\hat K_\pi Q_n$.

\section{Determinantal intertwining kernels}
\label{sec:determinant}
It is the aim of this section to show how the kernel $Q_n$ of $Y$ can be recovered
from any of the intertwining identities $P_n\hat K_\alpha=\hat K_\alpha Q_n$ and
$P_nK_\alpha=K_\alpha Q_n$,
where $P_n$ is the (known) Karlin-McGregor type kernel of the process $Z$.
In fact, we prove the stronger assertion that the
intertwining kernels $K_\alpha$ and $\hat K_\alpha$ are invertible.

In what follows, it is convenient to embed ${\mathbf T}_N^{N,\wedge}$ in a space parametrized by
an array of variables
$\mathbf{x}=(x^1,\ldots,x^N)$ with $x^k=(x_1^k,x_2^k,\ldots,x_k^k)\in {\mathbf Z}^k$,
such that the coordinates satisfy the inequalities
\[
x^k_k\le x_{k-1}^{k-1}\leq x^k_{k-1}\leq x^{k-1}_{k-2}
\le \ldots\leq x_2^k\leq x_1^{k-1}\leq x_1^k
\]
for $k=2,\ldots,N$. The pattern $\mathbf x$ corresponding to a given SSYT in ${\mathbf T}_N^{N,\wedge}$
is found by letting $x_i^j$
mark the position of the last entry in row $i$ whose label does not
exceed $j$.
We write $\mathbf{K}^N$ for the set of all $\mathbf x$ satisfying the above constraint,
and say that any $\mathbf x \in \mathbf{K}^N$ is a {\em Gelfand-Tsetlin (GT) pattern}.
In contrast to the tableau setting, it is not required that $\mathbf x$
has non-negative entries. For $\mathbf{x}\in\mathbf{K}^N$, we set
$\mathrm{sh}(\mathbf{x})=(x^N_1,x^N_2, \ldots, x_N^N)$, 
$\mathrm{ledge}(\mathbf{ x})= (x^1_1,\ldots,x^N_N)$, and
$\mathrm{redge}(\mathbf{x})=(x_1^1,\ldots,x_1^N)$; these definitions are consistent with 
those given for a SSYT in ${\mathbf T}_N^{N,\wedge}$.
For a vector $\alpha=(\alpha_1,\alpha_2, \ldots, \alpha_N)$ of weights,
we define the weight $\alpha^{\mathbf{x}}$ of a GT pattern $\mathbf{x}$ by
\[
\alpha^{\mathbf{x}}= \alpha_1^{x^1_1} \prod_{k=2}^N \alpha_k^{ \sum
x^k_i- \sum x^{k-1}_i},
\]
in accordance with our previous definition for tableaux in ${\mathbf T}_N^{N,\wedge}$.

Instead of studying $K_\alpha$ and $\hat K_\alpha$, it is equivalent but more convenient
to work with modified versions which do not involve Schur functions but
have a polynomial prefactor.
It is natural to start with the kernel $K_\alpha$,
since this is a `square' $W^N\times W^N$ matrix.
We define, suppressing the dependence on $\alpha$,
\begin{equation}
\label{deflambda}
\Lambda( z, y )=\alpha_1^{-y_1}\cdots \alpha_N^{-y_N}\sum_{{\mathbf x}\in\mathbf{K}^N; \,\mathrm{sh}(\mathbf{x})=z, \, \mathrm{ledge}({\mathbf x})=y} \alpha^{\mathbf{x}}.
\end{equation}
We first show that $\Lambda(z,y)$ can be written as a determinant.
\begin{proposition}
For $y,z\in W^N$, we have
\[
\Lambda(z,y)= \mathrm{det} \bigl\{ h_{z_i-y_j-i+j}^{(jN)}(\alpha) \bigr\}.
\]
\end{proposition}
\begin{proof}
The proof is a variant of a well-known argument using
non-intersecting lattice paths to derive the Jacobi-Trudi formulae,
see for example \cite{stanley:ec2}. Fix $z$ and $y$ belonging to
$W^N$ with $z_N=y_N$. Each Gelfand-Tsetlin pattern ${\mathbf
x}\in\mathbf{K}^N $ having $\mathrm{sh}(\mathbf{x})=z$ and
$\mathrm{ledge}({\mathbf x})=y$ can be encoded as a non-intersecting
$(N-1)$-tuple of paths $\bigl(P_1,P_2, \ldots, P_{N-1}\bigr)$  on
the edges of the square lattice with vertex set ${\mathbf Z}^2$.
Paths always traverse edges in the direction of increasing
co-ordinates: either `upwards' or `rightwards'. The path $P_k$
begins at the vertex $(y_k-k,k+1)$, ends at $(z_k-k,N)$,  and
contains the `horizontal' edges from $(x_k^{r-1}-k,r)$ to
$(x^r_k-k,r)$, for $r=k+1,k+2, \ldots, N$.  This correspondence
between patterns and paths is  a bijection and consequently we may
write the sum defining  $\Lambda(z,y)$ as
\[
\sum_{(P_1,\ldots, P_{N-1})} w(P_1)\cdots w(P_{N-1}),
\]
where the sum is over all non-intersecting paths which have starting points and end
points given in terms of $y$ and $z$ as above, and where the weight
$w(P_k)$ is defined to be
\[
\prod_{r=k+1}^N \alpha_r^{e(r)},
\]
with $e(r)$ denoting the number of horizontal edges at height $r$
contained in  the path $P_k$. By the Gessel-Viennot formula (see
Theorem 2.7.1 of \cite{stanley:ec1}), this sum for $\Lambda(z,y)$ is
equal to $\det(M)$ where $M$ is an $(N-1) \times (N-1)$ matrix with
$(i,j)$th entry given by $\sum_P w(P)$ where $P$ runs through all
paths  connecting $(y_j-j,j+1)$ to $(z_i-i,N)$. This latter quantity
is easily found to equal $h_{z_i-y_j-i+j}^{(jN)}(\alpha)$. The
proposition follows on noting that
$h_{z_i-y_N-i+N}^{(NN)}(\alpha)=0$ for $i=1,2,3, \ldots, N-1$ and
that $h_{z_N-y_N}^{(NN)}(\alpha)={\mathbf 1}(z_N=y_N)$.
\end{proof}

We will use the following identity that is easily checked by means of generating functions,
\begin{equation}
\label{heidentity}
\sum_{r\in \mathbf Z} (-1)^r e_r^{(iN)}(\alpha) h^{(jN)}_{n-r}(\alpha)=
\begin{cases}
h^{(ji)}_n(\alpha)    &  \text{ if } j \leq i, \\
(-1)^n e^{(ij)}_n(\alpha) & \text{ if } i\leq j.
\end{cases}
\end{equation}
Recall also the Cauchy-Binet formula
\begin{equation}
\sum_{z \in W^N} \mathrm{det} \left\{ \phi_i(z_j-j)\right\} \mathrm{det} \left\{ \psi_j(z_i-i)\right\} = \mathrm{det} \left\{  \sum_{z \in {\mathbf Z}} \phi_i(z)\psi_j(z)\right\}.
\end{equation}
\begin{proposition}
The $W^N\times W^N$ matrix $\Lambda$ is invertible. Its inverse is given by
the $W^N \times W^N$ matrix $\Pi$ defined as
\[
\Pi(y,z)= \mathrm{det} \bigl\{ (-1)^{y_i-z_j-i+j}e^{(iN)}_{y_i-z_j-i+j}(\alpha)\bigr\}.
\]
\end{proposition}
\begin{proof}
We first show that $\Pi$ is a left inverse of $\Lambda$.
From the Cauchy-Binet formula we deduce that
\begin{equation}
\label{identity}
\Pi \Lambda (y,y^\prime)= \mathrm{det} \bigl\{ m^{(ij)}_{y_i-y^\prime_j-i+j} \bigr\},
\end{equation}
where
\[
m^{(ij)}_n= \begin{cases}
  h^{(ji)}_{n} & \text{ if  }j \leq i, \\
  (-1)^n e^{(ij)}_n & \text{ if }  i\leq j.
\end{cases}
\]
Observe that $m^{(ij)}_r$ is zero if either $i \leq j$ and $r>j-i$ or if $i \geq j$ and $r<0$.
In particular the product of the diagonal elements appearing in the determinant on the
righthandside of \eqref{identity} is $ \prod_i {\mathbf 1}(y_i^\prime=y_i)$.
We need to show that  this is the only contribution to the determinant.
Suppose that $\pi$ is a permutation of $\{1,2,\ldots,N\}$ other than the identity,
and consider the product of the $(i, \pi(i))$th entries.
It easy to see that there must be integers $i_1$ and $i_2$ such that
$i_1<i_2$, $\pi(i_1)>i_1$, $\pi(i_2)<i_2$ and $\pi(i_1)> \pi(i_2)$.
Now the $(i_1, \pi(i_1))$th entry is zero if $y_{i_1}>y^\prime_{\pi(i_1)}$, and the
$(i_2, \pi(i_2))$th entry is zero if  $y_{i_2}\leq y^\prime_{\pi(i_2)}$.
But since $y_{i_1} \geq y_{i_2}$ and $y^\prime_{\pi(i_1)} \leq y^\prime_{\pi(i_2)}$,
at least one of these previous inequalities holds, and hence the product is zero.

The key ingredient for showing that $\Pi$ is a right inverse of $\Lambda$ is the
trivial identity
\[
\Pi(y,z')=
\sum_{\ell_1,\ldots, \ell_N\in \mathbf Z} (-1)^{\ell_1} e_{\ell_1}^{(0N)} \cdots
(-1)^{\ell_N} e_{\ell_N}^{(NN)}\det \left\{{\mathbf 1}(y_i-i=z'_j-j+\ell_i)\right\}.
\]
Observe that the sum over the $\ell$ is finite since the summand is zero unless
$0\le \ell_i\le N-i$ for all $i$.
We therefore have for any function $f$ on $\mathbf Z^{N}$ and $0\le \ell_i\le N-i$,
\begin{eqnarray*}
\lefteqn{\sum_{y\in W^N} f(y_1-1,\ldots,y_N-N)\det \left\{{\mathbf 1}(y_i-i=z'_j-j+\ell_i)\right\}}
\\&=&
\sum_{\sigma\in \mathcal S_{N}} \sgn(\sigma) f(z'_{\sigma(1)}+\ell_1,\ldots,z'_{\sigma(N)}-\sigma(N)+\ell_N)\\
&& \mbox{}
\times\sum_{y\in W^N} \prod_{m=1}^N {\mathbf 1}(y_m-m=z'_{\sigma(m)}-\sigma(m)+\ell_m)\\
&=&f(z'_1-1+\ell_1,\ldots,z'_{N}-N+\ell_N),
\end{eqnarray*}
where $\mathcal S_{N}$ is the set of permutations on $\{1,\ldots,N\}$.
After absorbing the sum over the $\ell$ in the determinant, we obtain
$\Lambda\Pi(z,z')= \mathrm{det} \{ m^{(ij)}_{z_i-z^\prime_j-i+j} \}={\mathbf 1}(z=z')$.
\end{proof}

\medskip
By virtue of this proposition, the intertwining $\Lambda Q= P \Lambda$ yields $Q=\Pi P \Lambda$,
and a straightforward computation using the Cauchy-Binet formula
and \eqref{heidentity} implies the following corollary.
Recall the definition of $f^{(ij)}_\alpha$ in (\ref{eq:defwij}).

\begin{corollary}
\label{cor:QrepLambda}
Suppose that $P$ is a $W^N\times W^N$ matrix of the form
\[
P(z,z^\prime)= \mathrm{det}\left\{ f({z^\prime_i-z_j-i+j}) \right\}
\]
for some function $f$ on $\mathbf Z$.
Suppose that $Q$ is another $W^N\times W^N$ matrix
and that the intertwining relation $P\Lambda = \Lambda Q$ holds. Then we have
for $y,y'\in W^N$,
\[
Q(y,y^\prime)=  \mathrm{det} \left\{ f^{(ij)}_\alpha({y^\prime_i-y_j-i+j}) \right\}.
\]
\end{corollary}

We will also consider a second intertwining kernel that arises by replacing the left edge
of the pattern with the right edge in \eqref{deflambda}: for $z\in W^N$, $y\in\hat W^N$, we set
\[
\hat{\Lambda}( z, y )=\alpha_1^{-y_1} \cdots \alpha_N^{-y_N}
\sum_{{\mathbf x}\in\mathbf{K}^N;\, \mathrm{sh}(\mathbf{x})=z, \, \mathrm{redge}({\mathbf x})=y} \alpha^{\mathbf{x}}.
\]
For  $\mathbf x \in {\mathbf K}^N$, define $\hat{\mathbf x} \in {\mathbf K}^N$ by $\hat{x}^k_i=-x^k_{k-i+1}$.
The correspondence $\mathbf x \mapsto \hat{\mathbf x}$ is bijective and
it is easily verified that $\alpha^{\hat{\mathbf x}}= \beta^{\mathbf x} $ where
$\beta=\alpha^{-1}=(\alpha_1^{-1}, \ldots, \alpha_N^{-1})$.
Using this correspondence and our results for $\Lambda$ we obtain the following.
\begin{proposition} For $z\in W^N$ and $y\in \hat W^N$, we have
\[
\hat{\Lambda}( z, y )= \mathrm{det} \bigl\{ h_{y_j-z_{N-i+1}-i+j}^{(jN)}(\alpha^{-1}) \bigr\}.
\]
Moreover, $\hat{\Lambda}$ is invertible with inverse $\hat \Pi$ given by, for $y\in \hat W^N$ and $z\in W^N$,
\[
\hat{\Pi}(y,z)= \mathrm{det} \bigl\{ (-1)^{z_{N-j+1}-y_i-i+j}e^{(iN)}_{z_{N-j+1}-y_i-i+j}(\alpha^{-1})\bigr\}.
\]
\end{proposition}

\medskip
The analogue of Corollary~\ref{cor:QrepLambda} follows immediately from this proposition.
\begin{corollary}
Suppose that $P$  is a $W^N\times W^N$ matrix of the form
\[
P(z,z^\prime)= \mathrm{det}\left\{ f({z^\prime_i-z_j-i+j}) \right\}
\]
for some function $f$ on $\mathbf Z$.
Suppose that $Q$ is a $\hat W^N\times \hat W^N$ matrix and that the
intertwining relation $P\hat{\Lambda} = \hat{\Lambda} Q$ holds. Then we have
for $y,y'\in \hat W^N$,
\[
Q(y,y^\prime)=  \mathrm{det} \left\{ \hat{f}^{(ij)}_{\alpha^{-1}}({y^\prime_i-y_j+i-j}) \right\}.
\]
\end{corollary}

\section*{Acknowledgements}
We thank Timo Sepp\"al\"ainen for his helpful comments. 
The work of ABD was supported by the Science Foundation Ireland, grant number SFI04/RP1/I512.

{\small
\bibliography{c:/bibdb}
\bibliographystyle{amsplain}
}

\end{document}